\documentclass[11pt]{article}
\usepackage{amsmath}
\usepackage{amsthm}
\usepackage{amssymb}
\usepackage{bm}
\usepackage{mathrsfs}
\usepackage{mathtools}
\usepackage{xcolor}

\definecolor{darkblue}{RGB}{2,58,141}
\definecolor{lightgray}{RGB}{229,235,244}

\usepackage{tcolorbox}
\tcbset{on line,arc=0pt,outer arc=0pt,colback=lightgray,colframe=darkblue,
width=\linewidth+3pt,boxsep=0pt,top=5pt,left=-8pt,right=-8pt,
boxrule=0pt,bottomrule=0.5pt,toprule=0.5pt}

\usepackage{geometry}
\usepackage[colorlinks=true,
linkcolor=blue,citecolor=blue,
urlcolor=blue]{hyperref}

\makeatletter
\def\@seccntDot{.}
\def\@seccntformat#1{\csname the#1\endcsname\@seccntDot\hskip 0.5em}
\renewcommand\section{\@startsection{section}{1}{\z@}%
{18\p@ \@plus 6\p@ \@minus 3\p@}%
{9\p@ \@plus 6\p@ \@minus 3\p@}%
{\large\bfseries\boldmath}}
\renewcommand\subsection{\@startsection{subsection}{2}{\z@}%
{12\p@ \@plus 6\p@ \@minus 3\p@}%
{3\p@ \@plus 6\p@ \@minus 3\p@}%
{\bfseries\boldmath}}
\renewcommand\subsubsection{\@startsection{subsubsection}{3}{\z@}%
{12\p@ \@plus 6\p@ \@minus 3\p@}%
{\p@}%
{\bfseries\boldmath}}
\makeatother

\usepackage{microtype}

\theoremstyle{plain}
\newtheorem{theorem}{Theorem}[section]
\newtheorem{lemma}{Lemma}[section]
\newtheorem{corollary}{Corollary}[section]

\theoremstyle{definition}

\newtheorem{remark}{Remark}[section]

\newtheorem{claim}{Claim}[section]

\numberwithin{equation}{section}
\allowdisplaybreaks
\DeclareMathOperator{\ex}{ex}
\DeclareMathOperator{\Ex}{Ex}

\title{Spectral Tur\'an Type Problems on Cancellative Hypergraphs}

\author{Zhenyu Ni\thanks{Department of Mathematics, Hainan University, Haikou 570228, P.R. China (\texttt{995264@hainanu.edu.cn}).}
\and Lele Liu\thanks{College of Science, University of Shanghai for Science and Technology, Shanghai 200093, 
P.R. China (\texttt{ahhylau@outlook.com}). This author is supported by the National Natural Science Foundation of China (No. 12001370).}
\thanks{Corresponding author.}
\and Liying Kang\thanks{Department of Mathematics, Shanghai University, Shanghai 200444, P.R. China (\texttt{lykang@shu.edu.cn}).
This author is supported by the National Natural Science Foundation of China (Nos. 11871329, 11971298)}
}

\date{}

\begin{document}
\maketitle

\begin{center}
\begin{tcolorbox}
\begin{abstract}
Let $G$ be a cancellative $3$-uniform hypergraph in which the symmetric difference of any two edges is not 
contained in a third one. Equivalently, a $3$-uniform hypergraph $G$ is cancellative if and only if $G$ is 
$\{F_4, F_5\}$-free, where $F_4 = \{abc, abd, bcd\}$ and $F_5 = \{abc, abd, cde\}$. A classical result in 
extremal combinatorics stated that the maximum size of a cancellative hypergraph is achieved by the balanced 
complete tripartite $3$-uniform hypergraph, which was firstly proved by Bollob\'as and later by Keevash and Mubayi.
In this paper, we consider spectral extremal problems for cancellative hypergraphs. More precisely, we 
determine the maximum $p$-spectral radius of cancellative $3$-uniform hypergraphs, and characterize the 
extremal hypergraph. As a by-product, we give an alternative proof of Bollob\'as' result from spectral viewpoint.
\par\vspace{2mm}

\par\vspace{2mm}

\end{abstract}
\end{tcolorbox}
\end{center}

\section{Introduction}

Consider an $r$-uniform hypergraph (or $r$-graph for brevity) $G$ and a family of $r$-graphs $\mathcal{F}$. 
We say $G$ is \emph{$\mathcal{F}$-free} if $G$ does not contain any member of $\mathcal{F}$ as a subhypergraph. 
The \emph{Tur\'an number} $\ex(n, \mathcal{F})$ is the maximum number of edges of an $\mathcal{F}$-free 
hypergraph on $n$ vertices. Determining Tur\'an numbers of graphs and hypergraphs is one of the central problems 
in extremal combinatorics. For graphs, the problem was asymptotically solved for all non-bipartite graphs 
by the celebrated Erd\H os-Stone-Simonovits Theorem. By contrast with the graph case, there is comparatively 
little understanding of the hypergraph Tur\'an number. We refer the reader to the surveys 
\cite{Furedi1991,Keevash2011,Mubayi-Verstraete2016}.

In this paper we consider spectral analogues of Tur\'an type problems for $r$-graphs. For $r=2$, the picture 
is relatively complete, due in large part to a longstanding project of Nikiforov, see e.g., \cite{Nikiforov2011} 
for details. However, for $r \geq 3$ there are very few known results. In \cite{Keevash2014}, Keevash-Lenz-Mubayi 
determine the maximum $p$-spectral radius of any $3$-graph on $n$ vertices not containing the Fano plane when 
$n$ is sufficiently large. They also obtain a $p$-spectral version of the Erd\H os-Ko-Rado theorem on 
$t$-intersecting $r$-graphs. Recently, Ellingham-Lu-Wang \cite{Ellingham-Lu-Wang2022} show that the $n$-vertex 
outerplanar $3$-graph of maximum spectral radius is the unique 3-graph whose shadow graph is the join of an 
isolated vertex and the path $P_{n-1}$. Gao-Chang-Hou \cite{GaoChangHou2022} study the extremal problem for 
$K_{r+1}^+$-free $r$-graphs among linear hypergraphs, where $K_{r+1}^+$ is obtained from the complete graph $K_{r+1}$ by 
enlarging each edge of $K_{r+1}$ with $r - 2$ new vertices disjoint from $V(K_{r+1})$ such that distinct edges of $K_{r+1}$ 
are enlarged by distinct vertices.

To state our results precisely, we need some basic definitions and notations. A $3$-graph is \emph{tripartite} 
or \emph{$3$-partite} if it has a vertex partition into three parts such that every edge has exactly one vertex 
in each part. Let $T_3(n)$ be the complete $3$-partite $3$-graph on $n$ vertices with part sizes
$\lfloor n/3\rfloor$, $\lfloor (n+1)/3\rfloor$, $\lfloor (n+2)/3\rfloor$, and $t_3(n)$ be the number of edges 
of $T_3(n)$. That is, 
\[ 
t_3(n) = \Big\lfloor \frac{n}{3} \Big\rfloor \cdot \Big\lfloor \frac{n+1}{3} \Big\rfloor \cdot \Big\lfloor \frac{n+2}{3} \Big\rfloor.
\]
We call an $r$-graph $G$ \emph{cancellative} if $G$ has the property that for any edges $A$, $B$, $C$ whenever 
$A\cup B = A\cup C$, we have $B=C$. Equivalently, $G$ is cancellative if $G$ has no three distinct triples 
$A$, $B$, $C$ satisfying $B\triangle C\subset A$, where $\triangle$ is the symmetric difference. For graphs, 
the condition is equivalent to saying that $G$ is triangle-free. Moving on to $3$-graphs, we observe that 
$B\Delta C\subset A$ can only occur when $|B\cap C|=2$ for $B\neq C$. This leads us to identify the two 
non-isomorphic configurations that are forbidden in a cancellative $3$-graph: $F_4 = \{abc, abd, bcd\}$ 
and $F_5 = \{abc, abd, cde\}$.

It is well-known that the study of Tur\'an numbers dates back to Mantel's theorem, which states that 
$\ex (n, K_3) = \lfloor n^2/4\rfloor$. As an extension of the problem to hypergraphs, Katona 
conjectured, and Bollob\'as \cite{Bollobas1974} proved the following result.

\begin{theorem}[\cite{Bollobas1974}]\label{thm:edge-extremal}
A cancellative $3$-graph on $n$ vertices has at most $t_3(n)$ edges, with equality only for $T_3(n)$.
\end{theorem}

In \cite{KeevashaMubayi2004}, Keevash and Mubayi presented a new proof of Bollob\'as' result, and further 
proved a stability theorem for cancellative hypergraphs. The main result of this paper is the following 
$p$-spectral analogues of Bollob\'as' result.

\begin{theorem}\label{thm:main}
Let $p\geq 1$ and $G$ be a cancellative $3$-graph on $n$ vertices. 
\begin{enumerate}
\item[$(1)$] If $p\geq 3$, then $\lambda^{(p)}(G)\leq \lambda^{(p)}(T_3(n))$, with equality if and only if $G = T_3(n)$.
\item[$(2)$] If $p=1$, then $\lambda^{(1)}(G)=1/9$.
\end{enumerate}
\end{theorem}

\section{Preliminaries}
\label{sec:preliminaries}

In this section we introduce definitions and notation that will be used throughout the paper, and give some 
preliminary lemmas.

Given an $r$-graph $G = (V(G), E(G))$ and a vertex $v$ of $G$. The \emph{link} $L_G(v)$ is the $(r-1)$-graph 
consisting of all $S\subset V(G)$ with $|S| = r-1$ and $S\cup \{v\} \in E(G)$. The \emph{degree} $d_G(v)$ of 
$v$ is the size of $L_G(v)$. As usual, we denote by $N_G(v)$ the neighbor of a vertex $v$, i.e., the set formed 
by all the vertices which form an edge with $v$. In the above mentioned notation, we will skip the index $G$ 
whenever $G$ is understood from the context.

The \emph{shadow graph} of $G$, denoted by $\partial(G)$, is the graph with $V(\partial(G)) = V(G)$ and 
$E(\partial(G))$ consisting of all pairs of vertices that belong to an edge of $G$, i.e., 
$E(\partial(G)) = \{e: |e| = 2,\, e\subseteq f\ \text{for some}\ f\in E(G)\}$. For more definitions and 
notation from hypergraph theory, see e.g., \cite{Bretto2013}.

For any real number $p\geq 1$, the $p$-spectral radius was introduced by Keevash, Lenz and Mubayi \cite{Keevash2014} 
and subsequently studied by Nikiforov \cite{Nikiforov2014:analytic-methods,Nikiforov2014:extremal-problems}. 
Let $G$ be an $r$-graph of order $n$, the polynomial form of $G$ is a multi-linear function 
$P_G(\bm{x}): \mathbb{R}^n\to\mathbb{R}$ defined for any vector $\bm{x}=(x_1,x_2,\ldots,x_n)^{\mathrm{T}}\in\mathbb{R}^n$ as
\[
P_G(\bm{x})=r\sum_{\{i_1,i_2,\ldots,i_r\}\in E(G)} x_{i_1}x_{i_2}\cdots x_{i_r}.
\]
The \emph{$p$-spectral radius}\footnote{We modified the definition of $p$-spectral radius by removing 
a constant factor $(r-1)!$ from \cite{Keevash2014}, so that the $p$-spectral radius is the same as the one 
in \cite{Cooper2012} when $p=r$. This is not essential and does not affect the results at all.} of $G$ is defined as%
\begin{equation}\label{eq:definition-p-spectral-radius}
\lambda^{(p)}(G):=\max_{\|\bm{x}\|_p=1} P_G(\bm{x}),
\end{equation}
where $\|\bm{x}\|_p:=(|x_1|^p+\cdots+|x_n|^p)^{1/p}$.

For any real number $p\geq 1$, we denote by $\mathbb{S}_{p,+}^{n-1}$ the set of all nonnegative real 
vectors $\bm{x}\in\mathbb{R}^n$ with $\|\bm{x}\|_p=1$. If $\bm{x}\in\mathbb{R}^n$ is a vector with 
$\|\bm{x}\|_p=1$ such that $\lambda^{(p)}(G)=P_G(\bm{x})$, then $\bm{x}$ is called an \emph{eigenvector} 
corresponding to $\lambda^{(p)}(G)$. Note that $P_G(\bm{x})$ can always reach its maximum at some nonnegative 
vectors. By Lagrange's method, we have the \emph{eigenequations} for $\lambda^{(p)}(G)$ and 
$\bm{x}\in\mathbb{S}_{p,+}^{n-1}$ as follows:
\begin{equation}\label{eq:eigenequation}
\lambda^{(p)}(G)x_i^{p-1} = \sum_{\{i,i_2,\ldots,i_r\}\in E(G)}x_{i_2}\cdots x_{i_r}
~~\text{for}\ x_i>0.
\end{equation}

It is worth mentioning that the $p$-spectral radius $\lambda^{(p)}(G)$ shows remarkable connections with some hypergraph
invariants. For instance, $\lambda^{(1)}(G)/r$ is the Lagrangian of $G$, $\lambda^{(r)}(G)$ is the usual
spectral radius introduced by Cooper and Dutle \cite{Cooper2012}, and $\lambda^{(\infty)}(G)/r$ is the
number of edges of $G$ (see \cite[Proposition 2.10]{Nikiforov2014:analytic-methods}).


Given two vertices $u$ and $v$, we say that $u$ and $v$ are \emph{equivalent} in $G$, in writing $u \sim v$, 
if transposing $u$ and $v$ and leaving the remaining vertices intact, we get an automorphism of $G$.

\begin{lemma}[\cite{Nikiforov2014:analytic-methods}]\label{lem:equivalent}
Let $G$ be a uniform hypergraph on $n$ vertices and $u\sim v$. If $p>1$ and $\bm{x}\in\mathbb{S}_p^{n-1}$
is an eigenvector to $\lambda^{(p)}(G)$, then $x_u = x_v$.
\end{lemma}

\section{Cancellative hypergraph of maximum $p$-spectral radius}

The aim of this section is to give a proof of Theorem \ref{thm:main}. We split it into 
Theorem \ref{thm:p-leq-3} -- Theorem \ref{thm:p=1}, which deal with $p=3$, $p>3$ and $p=1$, respectively.

\subsection{General properties on cancellative hypergraphs}

We start this subsection with a basic fact.

\begin{lemma}\label{lem:union-triangle-free}
Let $G$ be a cancellative hypergraph, and $u,v$ be adjacent vertices. Then $L(u)$ and $L(v)$ are edge-disjoint graphs.
\end{lemma}

\begin{proof}
Assume by contradiction that $e\in E(L(u)) \cap E(L(v))$. Since $u$ and $v$ are adjacent in $G$, we have 
$\{u,v\}\subset e_1\in E(G)$ for some edge $e_1$. Hence, $e_2 = e\cup\{u\}$, $e_3 = e\cup\{v\}$ and $e_1$ 
are three edges of $G$ such that $e_2\Delta e_3 \subset e_1$, a contradiction.
\end{proof}

Let $G$ be a $3$-graph and $v\in V(G)$. We denote by $E_v(G)$ the collection of edges of $G$ containing $v$, 
i.e., $E_v(G) = \{e: v\in e\in E(G)\}$. For a pair of vertices $u$ and $v$ in $G$, we denote by $T_{v}^{u}(G)$ 
a new $3$-graph with $V(T_{v}^{u}(G))=V(G)$ and 
\[ 
E(T_{v}^{u}(G)) = \big(E(G)\setminus E_v(G)\big) \cup \{(e\setminus \{u\}) \cup \{v\}: e\in E_u(G)\setminus E_v(G)\}.
\]

\begin{lemma}\label{lem:transfer-u-v}
Let $G$ be a cancellative $3$-graph. Then $T_v^u(G)$ is also cancellative for any $u,v\in V(G)$.
\end{lemma}

\begin{proof}
Suppose to the contrary that there exist three edges $e_1,e_2,e_3\in T_{v}^{u}(G)$ such that $e_1\triangle e_2 \subset e_3$. 
Recalling the definition of $T_{v}^{u}(G)$, we deduce that $u$, $v$ are non-adjacent in $T_{v}^{u}(G)$, and 
$( e\cup\{u\}) \setminus\{v\}\in E(G)$ for any $e\in E_v(T_{v}^{u}(G))$. On the other hand, since $G$ is cancellative, 
we have $v\in e_1\cup e_2\cup e_3$. Denote by $\alpha$ the number of edges $e_1$, $e_2$, $e_3$ containing $v$. It 
suffices to consider the following three cases.
	
{\bfseries Case 1.} $\alpha = 3$. We have $v\in e_1\cap e_2\cap e_3$. Hence, $e_1'=\left( e_1\cup\{u\}\right) \setminus\{v\}$, 
$e_2'=\left( e_2\cup\{u\}\right) \setminus\{v\}$ and $e_3'=\left( e_3\cup\{u\}\right) \setminus\{v\}$ are three edges 
in $G$ with $e_1'\triangle e_2'\subset e_3'$. This contradicts the fact that $G$ is cancellative.
	
{\bfseries Case 2.} $\alpha = 2$. Without loss of generality, we assume $v\in (e_1\cap e_2)\setminus e_3$ or 
$v\in (e_1\cap e_3)\setminus e_2$. If $v\in (e_1\cap e_2)\setminus e_3$, then $e_3\in E(G)$. It follows that 
$e_1' = (e_1\cup\{u\}) \setminus\{v\}$, $e_2' = (e_2\cup\{u\}) \setminus\{v\}$ and $e_3$ are three edges of $G$ 
with $e_1'\triangle e_2'\subset e_3$, which is a contradiction. If $v\in (e_1\cap e_3)\setminus e_2$, then $e_2\in E(G)$.
It follows that $e_1' = (e_1\cup\{u\}) \setminus\{v\}$, $e_2$ and $e_3' = (e_3\cup\{u\}) \setminus\{v\}$ are 
three edges of $G$ with $e_1'\triangle e_2\subset e_3'$, a contradiction.
	
{\bfseries Case 3.} $\alpha = 1$. Without loss of generality, we assume $v\in e_3\setminus(e_1\cup e_2)$. Then 
$e_1\in E(G)$ and $e_2\in E(G)$. We immediately obtain that $e_1$, $e_2$ and $e_3' = (e_3\cup\{u\}) \setminus\{v\}$ 
are three edges of $G$ with $e_1\triangle e_2\subset e_3'$. This is a contradiction and proves Lemma \ref{lem:transfer-u-v}.
\end{proof}

\begin{lemma}\label{lem:complete-3-partite-p-spectral-radius}
Let $p>1$ and $G$ be a complete $3$-partite $3$-graph. Then 
\[
\lambda^{(p)} (G) = \frac{(27\cdot |E(G)|)^{1-1/p}}{9}.
\]
\end{lemma}

\begin{proof}
Assume that $V_1$, $V_2$ and $V_3$ are the vertex classes of $G$ with $n_i: = |V_i|$ and $n_1\geq n_2\geq n_3$. 
Let $\bm{x}\in\mathbb{S}_{p,+}^{n-1}$ be an eigenvector corresponding to $\lambda^{(p)} (G)$. By 
Lemma \ref{lem:equivalent}, for $i=1,2,3$ we denote $a_i := x_v$ for $v\in V_i$, and set $\lambda := \lambda^{(p)} (G)$ 
for short. In light of eigenequation \eqref{eq:eigenequation}, we find that
\[
\begin{cases}
\lambda a_1^{p-1} = n_2n_3 a_2a_3, \\
\lambda a_2^{p-1} = n_1n_3 a_1a_3, \\
\lambda a_3^{p-1} = n_1n_2 a_1a_2,
\end{cases}
\]
from which we obtain that $a_i = (3n_i)^{-1/p}$, $i=1,2,3$. Therefore, 
\[
\lambda = \frac{(27\cdot n_1n_2n_3)^{1-1/p}}{9}
= \frac{(27\cdot |E(G)|)^{1-1/p}}{9}.
\]
This completes the proof of Lemma \ref{lem:complete-3-partite-p-spectral-radius}.
\end{proof}

\subsection{Extremal $p$-spectral radius of cancellative hypergraphs}

Let $\Ex_{sp}(n,\{F_4,F_5\})$ be the set of all $3$-graphs attaining the maximum $p$-spectral radius among 
cancellative hypergraphs on $n$ vertices. Given a vector $\bm{x}\in\mathbb{R}^n$ and a set $S\subset [n]: = \{1,2,\ldots,n\}$, 
we write $\bm{x} (S) := \prod_{i\in S} x_i$ for short. The \emph{support set} $S$ of a vector $\bm{x}$ is the index 
of non-zero elements in $\bm{x}$, i.e., $S=\{i\in [n]: x_i\neq 0\}$. Also, we denote by $x_{\min} := \min\{|x_i|: i\in [n]\}$ 
and $x_{\max} := \max\{|x_i|: i\in [n]\}$.

\begin{lemma}\label{lem:xu=xv}
Let $p > 1$, $G\in \Ex_{sp}(n,\{F_4,F_5\})$, and $\bm{x}\in\mathbb{S}_{p,+}^{n-1}$ be an eigenvector corresponding 
to $\lambda^{(p)}(G)$. If $u,v$ are two non-adjacent vertices, then $x_u = x_v$. 
\end{lemma}

\begin{proof}
Assume $u$ and $v$ are two non-adjacent vertices in $G$. Since $G$ is a cancellative $3$-graph, we have $T_u^v(G)$ 
is also cancellative by Lemma \ref{lem:transfer-u-v}. It follows from \eqref{eq:definition-p-spectral-radius} and 
\eqref{eq:eigenequation} that	
\begin{align*}
\lambda^{(p)} (T_u^v (G))
& \geq 3\sum_{e\in E(G)} \bm{x}(e) - 3 \sum_{e\in E_u(G)} \bm{x}(e) 
+ 3 \sum_{e\in E_v(G)} \bm{x}(e\setminus \{v\}) \cdot x_u \\
& = \lambda^{(p)} (G) - 3\lambda^{(p)} (G) x_u^p + 3\lambda^{(p)}(G) x_v^{p-1} x_u \\
& = \lambda^{(p)} (G) + 3\lambda^{(p)} (G) (x_v^{p-1} - x_u^{p-1}) \cdot x_u,
\end{align*}
which yields that $x_u \geq x_v$. Likewise, we also have $x_v \geq x_u$ by considering $T_v^u(G)$.
Hence, $x_u=x_v$, completing the proof of Lemma \ref{lem:xu=xv}.
\end{proof}

\begin{lemma}\label{lem:construction}
Let $p > 1$, $G\in \Ex_{sp}(n,\{F_4,F_5\})$, and $u,v$ be two non-adjacent vertices. Then there exists
a cancellative $3$-graph $H$ such that
\begin{equation}\label{eq:equality-link}
L_H(u) = L_H(v), ~~ \lambda^{(p)}(H) = \lambda^{(p)}(G), \ \text{and}\ d_H(w) \leq d_G(w),~~ w\in V(G). 
\end{equation}
\end{lemma}

\begin{proof}
Assume that $\bm{x}\in\mathbb{S}_{p,+}^{n-1}$ is an eigenvector corresponding to $\lambda^{(p)}(G)$. By 
Lemma \ref{lem:xu=xv}, $x_u = x_v$. Without loss of generality, we assume $d_G(u) \geq d_G(v)$. In view 
of \eqref{eq:definition-p-spectral-radius} and \eqref{eq:eigenequation}, we have
\begin{align*}
\lambda^{(p)}(T_u^v (G))
& \geq 3\sum_{e\in E(G)} \bm{x}(e) - 3 \sum_{e\in E_u(G)} \bm{x}(e) 
+ 3 \sum_{e\in E_v(G)} \bm{x}(e\setminus \{v\}) \cdot x_u \\
& = \lambda^{(p)}(G) + 3\lambda^{(p)}(G) (x_v^{p-1} - x_u^{p-1}) \cdot x_u  \\
& = \lambda^{(p)}(G).
\end{align*} 
Observe that $T_u^v (G)$ is a cancellative $3$-graph and $G\in \Ex_{sp}(n,\{F_4,F_5\})$. We immediately 
obtain that $\lambda^{(p)}(T_u^v (G)) = \lambda^{(p)}(G)$. It is straightforward to check that 
$H:= T_u^v(G)$ is a cancellative $3$-graph satisfying \eqref{eq:equality-link}, as desired.
\end{proof}

Next, we give an estimation on the entries of eigenvectors corresponding to $\lambda^{(p)}(G)$.

\begin{lemma}\label{lem:xmin-xmax}
Let $G\in \Ex_{sp}(n,\{F_4,F_5\})$ and $\bm{x}\in\mathbb{S}_{p,+}^{n-1}$ be an eigenvector corresponding 
to $\lambda^{(p)}(G)$. If $1 < p\leq 3$, then
\[ 
x_{\min} > \Big(\frac{3}{4}\Big)^{2/(p-1)} \cdot x_{\max}.
\]
\end{lemma}

\begin{proof}
Suppose to the contrary that $x_{\min} \leq \big(\frac{3}{4}\big)^{2/(p-1)} \cdot x_{\max}$. Let $u$ and 
$v$ be two vertices such that $x_u = x_{\min}$ and $x_v = x_{\max} > 0$. Then we have 
\[ 
\left( 1 + \frac{x_u}{x_v}\right) \left(\frac{x_u}{x_v}\right)^{p-1} \leq 
\bigg( 1 + \left(\frac{3}{4}\right)^{2/(p-1)}\bigg) \left(\frac{3}{4}\right)^2 \leq \frac{7}{4}\cdot\frac{9}{16} < 1,
\]
which implies that 
\begin{equation}\label{eq:xv3-xu3}
x_v^p - x_u^p > x_u^{p-1} x_v.
\end{equation}
On the other hand, by eigenequations we have 
\begin{equation}\label{eq:Ev-Eu}
\sum_{e\in E_v(G)\setminus E_u(G)} \bm{x}(e) \geq \lambda^{(p)}(G) (x_v^p - x_u^p).
\end{equation}
Now, we consider the cancellative $3$-graph $T_u^v(G)$. In light of \eqref{eq:definition-p-spectral-radius} and 
\eqref{eq:Ev-Eu}, we have 
\begin{align*}
\lambda^{(p)}(T_u^v (G)) 
& \geq 3 \sum_{e\in E(G)} \bm{x}(e) - 3 \sum_{e\in E_u(G)} \bm{x}(e)
+ 3 \sum_{e\in E_v(G)\setminus E_u(G)} \bm{x}(e\setminus \{v\}) \cdot x_u \\
& \geq \lambda^{(p)}(G) - 3\lambda^{(p)}(G) x_u^p + 3\lambda^{p}(G) (x_v^p - x_u^p)\cdot\frac{x_u}{x_v} \\
& > \lambda^{(p)}(G) + 3\lambda^{(p)}(G) \Big(-x_u^p + x_u^{p-1} x_v\cdot\frac{x_u}{x_v}\Big) \\
& =\lambda^{(p)}(G),
\end{align*} 
where the third inequality is due to \eqref{eq:xv3-xu3}. This contradicts the fact that $G$ has maximum 
$p$-spectral radius over all cancellative hypergraphs. 
\end{proof}

Now, we are ready to give a proof of Theorem \ref{thm:main} for $p = 3$.

\begin{theorem}\label{thm:p-leq-3}
Let $G$ be a cancellative $3$-graph on $n$ vertices. Then $\lambda^{(3)} (G)\leq\lambda^{(3)}(T_3(n))$ 
with equality if and only if $G = T_3(n)$.
\end{theorem}

\begin{proof}
According to Lemma \ref{lem:construction}, we assume that $G^*\in \Ex_{sp}(n,\{F_4,F_5\})$ is a $3$-graph such that 
$L_{G^*}(u) = L_{G^*} (v)$ for any non-adjacent vertices $u$ and $v$. 

Our first goal is to show $G^* = T_3(n)$ by Claim \ref{claim:degree} -- Claim \ref{claim:emptyset}.
Assume that $\bm{x}\in\mathbb{S}_{3,+}^{n-1}$ is an eigenvector corresponding to $\lambda^{(3)}(G^*)$; $u_1$ is 
a vertex in $G^*$ such that $x_{u_1} = x_{\max}$ and $u_2$ is a vertex with $x_{u_2}=\max\{x_v: v\in N_{G^*}(u_1)\}$. 
Let $U_1 := V(G^*)\setminus N_{G^*}(u_1)$ and $U_2 := V(G^*)\setminus N_{G^*}(u_2)$. Since $u_2\in V(G^*)\setminus U_1$, 
there exists a vertex $u_3$ such that $\{u_1,u_2,u_3\}\in E(G^*)$. Let $U_3=V(G^*)\setminus N_{G^*}(u_3)$. Recall 
that for any non-adjacent vertices $u$ and $v$ we have $L_{G^*}(u) = L_{G^*}(v)$. Hence, the sets $U_1$, $U_2$ 
and $U_3$ are well-defined.

\begin{claim}\label{claim:degree}
The following statements hold: 
\begin{enumerate}	
\item[$(1)$] $d_{G^*} (u_1) > n(n-1)/9$;

\item[$(2)$] $d_{G^*} (u_2) > n(n-1)/12$;

\item[$(3)$] $d_{G^*} (v) > n(n-1)/16$, $v\in V(G^*)$.
\end{enumerate}
\end{claim}

\noindent {\it Proof of Claim \ref{claim:degree}}.
Since $T_3(n)$ is a cancellative $3$-graph, it follows from Lemma \ref{lem:complete-3-partite-p-spectral-radius} that
\[ 
\lambda^{(3)}(G^*) \geq \lambda^{(3)} (T_3(n)) = \frac{\big(27\cdot t_3(n)\big)^{2/3}}{9}.
\]
By simple algebra we see 
\begin{equation}\label{eq:lower-bound-p-spectral-radius}
\lambda^{(3)} (G^*) \geq \frac{\big((n-2)(n+1)^2\big)^{2/3}}{9} > \frac{n(n-1)}{9}.
\end{equation}

(1). By eigenequation with respect to $u_1$, we have  
\[
\lambda^{(3)}(G^*) x_{u_1}^2 = \sum_{\{u_1,i,j\}\in E(G^*)} x_ix_j \leq d_{G^*}(u_1) x_{u_1}^2. 
\]
Combining with \eqref{eq:lower-bound-p-spectral-radius}, we get
\begin{equation}\label{eq:degree-u1}
d_{G^*}(u_1) \geq \lambda^{(3)} (G^*) > \frac{n(n-1)}{9}.
\end{equation}

(2). Observe that the definition of $U_1$, and $L_{G^*}(u) = L_{G^*} (v)$ for any pair $u,v\in U_1$. 
We immediately obtain that $|(e\setminus \{u_2\}) \cap U_1| \leq 1$ for each $e\in E_{u_2}(G^*)$. It 
follows from $x_{u_2}=\max\{x_v: v\in V(G^*)\setminus U_1\}$ that
\[
\lambda^{(3)} (G^*) x_{u_2}^2 = \sum_{\{u_2,i,j\}\in E(G^*)} x_ix_j \leq d_{G^*}(u_2) x_{u_1}x_{u_2},
\]
which, together with Lemma \ref{lem:xmin-xmax} for $p = 3$, gives
\begin{align*}
d_{G^*}(u_2) 
& \geq \frac{x_{u_2}}{x_{u_1}} \cdot \lambda^{(3)} (G^*) \\
& \geq \frac{3}{4} \cdot \lambda^{(3)} (G^*) \\
& > \frac{1}{12} n(n-1).
\end{align*}
The last inequality is due to \eqref{eq:lower-bound-p-spectral-radius}.

(3). Let $v$ be an arbitrary vertex in $V(G^*)$. Then 
\[
\lambda^{(3)}(G^*) x_v^2 = \sum_{\{v,i,j\}\in E(G^*)} x_i x_j \leq d_{G^*}(v) x_{u_1}^2 . 
\]
Hence, by Lemma \ref{lem:xmin-xmax} and \eqref{eq:lower-bound-p-spectral-radius} we have 
\[ 
d_{G^*}(v) \geq \Big(\frac{x_v}{x_{u_1}}\Big)^2 \cdot \lambda^{(3)}(G^*) > \frac{1}{16}n(n-1),
\]
as desired. 
\hfill $\Box$ \par\vspace{2mm}

Next, we consider the graph $H = L_{G^*}(u_1)\cup L_{G^*}(u_2)\cup L_{G^*}(u_3)$. Let $\phi: E(H) \to [3]$ 
be a mapping such that $\phi(f)=i$ if $f\in L_{G^*}(u_i)$, $i\in [3]$. By Lemma \ref{lem:union-triangle-free}, 
$\phi$ is an edge coloring of $H$. For convenience, we denote $L := V(G^*)\setminus(U_1\cup U_2\cup U_3)$.

\begin{claim}\label{claim:color}
If $L\neq \emptyset$, then there is no rainbow star $K_{1,3}$ in the induced subgraph $H[L]$ with the coloring $\phi$.
\end{claim}
	
\noindent {\it Proof of Claim \ref{claim:color}}.
Suppose to the contrary that there exist $v_1,v_2,v_3,v_4\in L$ with $\phi(v_1v_2) = 1$, $\phi(v_1v_3) = 2$ 
and $\phi(v_1v_4) = 3$. We first show that $\{v_1,v_2,v_3,v_4\}$ induced a clique in $\partial(G^*)$ by 
contradiction. Without loss of generality, we assume $v_2v_3\notin E(\partial(G^*))$. Then $L_{G^*}(v_2)=L_{G^*}(v_3)$.
Since $\phi(v_1v_2)=1$ and $\phi(v_1v_3)=2$, we have $\{u_1,v_1,v_2\}\in E(G^*)$ and $\{u_2,v_1,v_3\}\in E(G^*)$.
This implies that $e_1=\{u_1,u_2,u_3\}$,  $e_2=\{u_1,v_1,v_2\}$ and  $e_3=\{u_2,v_1,v_2\}$ are three edges 
in $G^*$ with $e_2\triangle e_3\subset e_1$, which is impossible. 
     	
On the other hand, since $L = V(G^*)\setminus(U_1\cup U_2\cup U_3)$, we have $v_iu_j\in E(\partial(G^*))$ 
for any $i\in [4]$, $j\in [3]$. Therefore, every pair of vertices in $\{v_1,v_2,v_3,v_4,u_1,u_2,u_3\}$ 
is contained in an edge of $G^*$. Consider the graph 
\[
H' := \bigg(\bigcup_{i=1}^3 L_{G^*} (u_i)\bigg) \bigcup \bigg(\bigcup_{i=1}^4 L_{G^*} (v_i)\bigg).
\]
By Claim \ref{claim:degree}, we have 
\begin{align*}
|E(H')| & =\sum_{1\le i\le 3}d_{G^*}(u_i)+\sum_{1\le j\le 4} d_{G^*}(v_j) \\
& > \left(1 + \frac{3}{4} + 5\times\frac{9}{16}\right) \cdot \frac{1}{9}n(n-1) \\
& =\frac{73}{144}n(n-1) \\
& >\binom{n}{2},
\end{align*} 
a contradiction completing the proof of Claim \ref{claim:color}. \hfill $\Box$

\begin{claim}\label{claim:emptyset}
$L = \emptyset$.  
\end{claim}

\noindent {\it Proof of Claim \ref{claim:emptyset}}.
Suppose to the contrary that $L\neq\emptyset$. For $i=1,2,3$, let $L_i$ be the set of vertices in $L$ which 
is not contained in an edge with coloring $i$. By Claim \ref{claim:color}, we have $L = L_1\cup L_2\cup L_3$. 
Without loss of generality, we assume $L_1\neq\emptyset$. Let $w$ be a vertex in $L_1$. Then there exists 
an edge $f$ in $G^*$ such that $f=\{u_1,w,w'\}$, where $w'\in U_2\cup U_3$. 
If $w'\in U_2$, then $f'=\{u_1,u_3,w'\}\in E(G^*)$. 
Since $G^*$ is cancellative, $w$ is not a neighbor of $u_3$ in $G^*$.
This implies that $w\in U_3$, a contradiction to $w\in L$.
Similarly, if $w'\in U_3$, then $w\in U_2$, which is also a contradiction.
\hfill $\Box$ \par\vspace{3mm}

Now, we continue our proof. By Claim \ref{claim:emptyset}, we immediately obtain that $G^*$ is a complete 
$3$-partite $3$-graph with vertex classes $U_1$, $U_2$ and $U_3$. Hence, $G^* = T_3(n)$ by 
Lemma \ref{lem:complete-3-partite-p-spectral-radius}. 

Finally, it is enough to show that $G = T_3(n)$ for any $G\in \Ex_{sp}(n, \{F_4,F_5\})$. 
According to Lemma \ref{lem:construction} and Claim \ref{claim:emptyset}, we can transfer $G$ to the complete 
$3$-partite $3$-graph $T_3(n)$ by a sequence of switchings $T_u^v(\,\cdot\,)$ that keeping the spectral 
radius unchanged. Let $T_1,\ldots,T_s$ be such a sequence of switchings $T_u^v(\,\cdot\,)$ which turn $G$ into $T_3(n)$. 
Consider the $3$-graphs $G = G_0,G_1,\ldots,G_s = T_3(n)$ in which $G_i$ is obtained from $G_{i-1}$ 
by applying $T_i$. Let $\bm{z}\in\mathbb{S}_{3,+}^{n-1}$ be an eigenvector corresponding to $\lambda^{(3)}(G_{s-1})$ and 
$T_u^v(G_{s-1}) = T_3(n)$, and denote
\[ 
A:= V(G_{s-1})\setminus \big(N_{G_{s-1}}(v) \cup \{u\} \cup \{v\}\big).
\]
Hence, we have $L_{G_{s-1}}(w) = L_{G_{s-1}}(v) = L_{T_3(n)}(v)$ for each $w\in A$. In what follows, we shall 
prove $L_{G_{s-1}}(u) = L_{G_{s-1}}(v)$, and therefore $G_{s-1} = T_3(n)$. If $L_{G_{s-1}}(u) \neq L_{G_{s-1}}(v)$, 
there exists an edge $e = v_1v_2\in L_{G_{s-1}}(u)\setminus L_{G_{s-1}}(v)$ since $z_u = z_v$ by Lemma \ref{lem:xu=xv}.
Let $M_1$ and $M_2$ be two subsets of $V(G_{s-1})$ such that $M_1\cup M_2 = N_{G_{s-1}}(v)$ and 
$L_{G_{s-1}}(v) = K_{|M_1|,|M_2|}$. If $\{v_1, v_2\}\subset N_{G_{s-1}}(v)$, then $\{v_1, v_2\}\subset M_1$
or $\{v_1, v_2\}\subset M_2$. It follows that there exists a vertex $w\in N_{G_{s-1}}(v)$ such that 
$f_1:= \{v,w,v_1\}\in E(G_{s-1})$ and $f_2:= \{v,w,v_2\}\in E(G_{s-1})$. However, 
$f_1\Delta f_2\subset \{u,v_1,v_2\}\in E(G_{s-1})$, a contradiction. So we obtain $\{v_1,v_2\}\cap A \neq \emptyset$. 
Without loss of generality, we assume $v_1\in A$. Then $L_{G_{s-1}}(v_1) = L_{G_{s-1}}(v)$, i.e., 
$uv_2\in L_{G_{s-1}}(v)$. Thus, $u\in N_{G_{s-1}}(v)$, a contradiction. This implies that $G_{s-1} = T_3(n)$. 
Likewise, $G_{i-1} = G_i$ for each $i\in [s-1]$, and therefore $G = T_3(n)$. This completes the proof of the theorem.
\end{proof}

According to Theorem \ref{thm:p-leq-3}, we can give an alternative proof of Bollob\'as' result for 
$n \equiv{0} \pmod{3}$.

\begin{corollary}\label{coro:edge-extremal}
Let $G$ be a cancellative $3$-graph on $n$ vertices with $n \equiv{0} \pmod{3}$. Then 
$|E(G)| \leq t_3(n)$ with equality if and only if $G = T_3(n)$. 
\end{corollary}

\begin{proof}
Denote by $\bm{z}$ the all-ones vector of dimension $n$.
In view of \eqref{eq:definition-p-spectral-radius}, we deduce that
\[
\lambda^{(3)} (G) \geq \frac{P_G(\bm{z})}{\|\bm{z}\|_3^3} = \frac{3 |E(G)|}{n}.
\]
On the other hand, by Theorem \ref{thm:p-leq-3} we have
\[
\lambda^{(3)} (G) \leq \lambda^{(3)}(T_3(n)) = (t_3(n))^{2/3}.
\]
As a consequence, 
\[
|E(G)| \leq \frac{n}{3}\cdot (t_3(n))^{2/3} = t_3(n).
\]
Equality may occur only if $\lambda^{(3)}(G) = (t_3(n))^{2/3} = \lambda^{(3)}(T_3(n))$, and therefore 
$G = T_3(n)$ by Theorem \ref{thm:p-leq-3}.
\end{proof}

Next, we will prove Theorem \ref{thm:main} for the case $p>3$ as stated in Theorem \ref{thm:p>3}.

\begin{lemma}[\cite{Nikiforov2014:analytic-methods}]\label{lem:nonincreasing}
Let $p\geq 1$ and $G$ be an $r$-graph with $m$ edges. Then the function 
\[
f_G(p) := \bigg(\frac{\lambda^{(p)} (G)}{rm}\bigg)^p
\]
is non-increasing in $p$.
\end{lemma}

\begin{theorem}\label{thm:p>3}
Let $p>3$ and $G$ be a cancellative $3$-graph on $n$ vertices. Then $\lambda^{(p)} (G)\leq\lambda^{(p)}(T_3(n))$ 
with equality if and only if $G = T_3(n)$.
\end{theorem}

\begin{proof}
Assume that $p>3$ and $G$ is a $3$-graph in $\Ex_{sp} (n, \{F_4,F_5\})$ with $m$ edges. It is enough to show that 
$G = T_3(n)$. By Lemma \ref{lem:nonincreasing}, we have 
\[
\bigg(\frac{\lambda^{(p)} (G)}{3m}\bigg)^p \leq \bigg(\frac{\lambda^{(3)} (G)}{3m}\bigg)^3,
\]
which, together with $\lambda^{(3)} (G) \leq (t_3(n))^{2/3}$ by Theorem \ref{thm:p-leq-3}, gives
\[
\lambda^{(p)} (G) \leq (3m)^{1-3/p} \cdot (\lambda^{(3)}(G))^{3/p} \leq (3m)^{1-3/p} \cdot (t_3(n))^{2/p}. 
\]
On the other hand, we have 
\[
\lambda^{(p)} (G) \geq \lambda^{(p)} (T_3(n)) = \frac{\big(27\cdot t_3(n)\big)^{1-1/p}}{9}.
\]
We immediately obtain $m \geq t_3(n)$. The result follows from Theorem \ref{thm:edge-extremal}.
\end{proof}

Finally, we shall give a proof of Theorem \ref{thm:main} for the remaining case $p = 1$. In what follows,
we always assume that $\bm{x}\in\mathbb{S}_{1,+}^{n-1}$ is an eigenvector such that $\bm{x}$ has the minimum 
possible number of non-zero entries among all eigenvectors corresponding to $\lambda^{(1)}(G)$.
Before continuing, we need the following result.

\begin{lemma}[\cite{FrankRodl1984}]\label{lem:cover}
Let $G$ be an $r$-graph and $S$ be the support set of $\bm{x}$. Then for each pair vertices $u$ and $v$ in $S$, 
there is an edge in $G[S]$ containing both $u$ and $v$.
\end{lemma}

\begin{theorem}\label{thm:p=1}
Let $G$ be a cancellative $3$-graph. Then $\lambda^{(1)} (G) = 1/9$. 
\end{theorem}

\begin{proof}
Assume that $G$ is a cancellative $3$-graph with support set $S$. Let $H:=G[S]$. By Lemma \ref{lem:cover},
for any $u,v\in S$ there is an edge in $H$ containing both $u$ and $v$. Hence, for any two edges 
for each pair of edges of $H$ has at most one common vertex by $H$ being cancellative. So the shadow graph 
of $H$ is the complete graph $K_{|S|}$. Since $H$ is cancellative, the link graphs $L_H(u)$ and $L_H(v)$ are
edge-disjoint graphs for any distinct vertices $u,v\in S$. It follows from \eqref{eq:eigenequation} that
\begin{equation}\label{eq:sum-eigenequations}
|S|\cdot \lambda^{(1)}(G) = \sum_{uv\in E(\partial (H))} x_ux_v \leq \frac{1}{2} \Big(1-\frac{1}{|S|}\Big),
\end{equation}
where the last inequality follows from Motzkin--Straus Theorem \cite{MotzkinStraus1965}. On the other hand,
set 
\[
z_v = 
\begin{cases}
1/|S|, & v\in S, \\
0, & \text{otherwise}.
\end{cases}
\]
We immediately have 
\[
\lambda^{(1)}(G) \geq 3 \sum_{e\in E(H)} \bm{z} (e) 
= \sum_{v\in V(H)} \bigg(z_v\cdot\sum_{f\in L_H(v)} \bm{z} (f)\bigg) = \frac{|S|-1}{2|S|^2},
\]
where the last inequality follows from the fact that $d_H(v)=(|S|-1)/2$ for $v\in V(H)$.
Combining with \eqref{eq:sum-eigenequations} we get 
\[
\lambda^{(1)}(G) = \frac{|S|-1}{2|S|^2}.
\]
Clearly, $(|S|-1)/|S|^2$ attains its maximum at $|S|=3$ when $|S|\geq 3$. Hence, we see $\lambda^{(1)}(G)\leq 1/9$.
Finally, noting that $\lambda^{(1)}(G)$ is at least the Lagrangian of an edge $K_3^{(3)}$, i.e.,
\[
\lambda^{(1)}(G) \geq\lambda^{(1)}(K_3^{3}) = \frac{1}{9},
\]
we obtain $\lambda^{(1)}(G)=1/9$, as desired.
\end{proof}

\begin{remark}
For an $r$-graph $G$ on $n$ vertices, it is well-known that $\lambda^{(1)}(G)/r$ is the Lagrangian 
of $G$. In \cite{YanPeng2019}, Yan and Peng present a tight upper bound on $\lambda^{(1)}(G)$ 
for $F_5$-free $3$-graphs, see \cite{YanPeng2019} for details. 
\end{remark}

\end{document}